\documentclass[11pt,pdftex]{amsart}
\usepackage{amsrefs} 
\usepackage{txfonts,mathptmx}
\usepackage{mathrsfs} 

\usepackage{tikz-cd}

\numberwithin{equation}{section}

\newtheorem{thm}{Theorem}[section]
\newtheorem{lem}[thm]{Lemma}
\newtheorem{pro}[thm]{Proposition}
\newtheorem{cor}[thm]{Corollary}

\theoremstyle{remark}

\newtheorem*{notation}{Notation}
\newtheorem{rem}[thm]{Remark}
\newtheorem*{ack}{Acknowledgement}

\newcommand{\adj}{{\operatorname{Ad}}}
\newcommand{\qdl}{\ast}
\newcommand{\str}{\phi}
\newcommand{\straa}{\psi}

\newcommand{\cq}{Q_W}

\newcommand{\Q}{\mathbb{Q}}

\newcommand{\Z}{\mathbb{Z}}

\newcommand{\C}{\mathbb{C}}

\newcommand{\map}{\rightarrow}

\newcommand{\rank}{{\operatorname{rank}}}

\newcommand{\Hom}{\operatorname{Hom}}

\newcommand{\symqdl}{Q_{\mathfrak{S}_n}}
\newcommand{\orbit}{\mathcal{O}_W}
\newcommand{\rep}{\mathcal{R}_W}
\newcommand{\nc}{c(W)}
\newcommand{\Ab}{\mathrm{Ab}}
\newcommand{\Cw}{C_W}
\newcommand{\rs}{\Phi_W}

\begin{document}
\title{
The adjoint group of a Coxeter quandle}
\author[T. Akita]{Toshiyuki Akita}
\address{Department of Mathematics, Faculty of Science, Hokkaido University,
Sapporo, 060-0810 Japan}
\email{akita@math.sci.hokudai.ac.jp}

\maketitle

\begin{abstract}
We give explicit descriptions of the adjoint group $\adj(\cq)$ of the 
Coxeter quandle $Q_W$ associated with an arbitrary Coxeter group $W$. 
The adjoint group $\adj(\cq)$ turns out to be an intermediate group
between $W$ and the corresponding Artin group $A_W$,
and fits into a central extension of $W$ by a
finitely generated free abelian group.
We construct $2$-cocycles of $W$ corresponding to the central extension.
In addition, 
we prove that the commutator subgroup of
the adjoint group $\adj(\cq)$ is  isomorphic to the commutator subgroup
of $W$.
Finally, the root system $\rs$ associated with a Coxeter group $W$ turns out to be a rack.
We prove that the adjoint group $\adj(\rs)$
of $\rs$ is isomorphic to the adjoint group of $Q_W$.
\end{abstract}

\section{Introduction}
A nonempty set $Q$ equipped with a binary operation $Q\times Q\map Q$,
$(x,y)\mapsto x\qdl y$ is called a \emph{quandle}
if it satisfies the following three conditions:
\begin{enumerate}
\item $x\qdl x=x$ $(x\in Q)$,
\item $(x\qdl y)\qdl z=(x\qdl z)\qdl (y\qdl z)$ $(x,y,z\in Q)$,
\item For all $y\in Q$,
the map $Q\map Q$ defined by $x\mapsto x\qdl y$ is bijective.
\end{enumerate}
If $X$ satisfies (2) and (3) but not necessarily (1), 
then $X$ is called a \emph{rack}.
Quandles and racks  have been studied in low dimensional topology as well as in Hopf algebras
(see Nosaka \cite{nosaka-book} and Andruskiewitsch-Gra\~na
\cite{MR1994219} for instance).
To any quandle or rack $Q$ one can associate a group $\adj(Q)$ called the
\emph{adjoint group} of $Q$ (also called  the associated group or the enveloping group in the
literature). It is defined by the presentation
\[
\adj(Q):=\langle e_x\ (x\in Q)\mid e_y^{-1}e_x e_y=e_{x\qdl y}\ (x,y\in Q)\rangle.
\]
The assignment $Q\mapsto\adj(Q)$ is a functor from the category of quandles
or racks to the category of groups.
Although adjoint groups play important roles in the study of quandles and racks, 
not much is known about the structure of them, partly because
the definition of $\adj(Q)$ by a possibly infinite
presentation is difficult to work with in explicit calculations.
We refer Eisermann \cite{MR3205568} and 
Nosaka \cite{arXiv:1505.03077,nosaka-book} for generality of adjoint groups, and
\cite{arXiv:1011.1587,MR3205568,arXiv:1505.03077,nosaka-book} for  
descriptions of adjoint groups of certain classes of quandles.

In this paper, we will study the adjoint group of a Coxeter quandle.
Let $(W,S)$ be a Coxeter system, a pair of a Coxeter group $W$ and the set $S$ of
Coxeter generators of $W$.
Following Nosaka \cite{arXiv:1505.03077},
we define the \emph{Coxeter quandle} $\cq$ associated with $(W,S)$ to be the set
of all reflections of $W$:
\[
\cq:=\bigcup_{w\in W}w^{-1}Sw.
\]
The quandle operation is given by the conjugation $x\ast y:=y^{-1}xy=yxy$.
The symmetric group $\mathfrak{S}_n$ of $n$ letters is a Coxeter group (of type $A_{n-1}$),
and the associated Coxeter quandle $\symqdl$ is nothing but the set of all transpositions.
In their paper \cite{MR2799090},
Andruskiewitsch-Fantino-Garc{\'{\i}}a-Vendramin obtained remarkable
results concerning with $\adj(\symqdl)$.
Namely, they proved that $\adj(\symqdl)$ is an intermediate group between $\mathfrak{S}_n$
and the braid group $B_n$ of $n$ strands, in the sense that
the canonical projection $B_n\twoheadrightarrow\mathfrak{S}_n$ splits
into a sequence of surjections 
$B_n\twoheadrightarrow\adj(\symqdl)\twoheadrightarrow\mathfrak{S}_n$,
and that $\adj(\symqdl)$ fits into a central extension of the form
\begin{equation}\label{eq:sym-central}
0\map \Z\map\adj(\symqdl)\overset{}{\map} \mathfrak{S}_n\map 1.
\end{equation}
See \cite{MR2799090}*{Proposition 3.2} and the proofs therein.
Furthermore, the central extension (\ref{eq:sym-central}) turns out to be the unique nontrivial central
extension of $\mathfrak{S}_n$ by $\Z$  (see Corollary \ref{cor:unique-cent}).

The primary purpose of this paper is to generalize those results to
arbitrary Coxeter quandles.
We will show that $\adj(\cq)$ is an intermediate group between $W$
and the Artin group $A_W$ associated with $W$ (Proposition \ref{pro-artin-to-adj}), 
and that $\adj(\cq)$ fits into a central extension of the form
\begin{equation}\label{eq:central}
0\map \Z^{\nc}\map\adj(\cq)\overset{\str}{\map} W\map 1,
\end{equation}
where $\nc$ is the number of conjugacy classes of elements in $\cq$
(Theorem \ref{thm-main}).
As a byproduct of Theorem \ref{thm-main}, we will determine the rational
cohomology ring of $\adj(\cq)$ (Corollary \ref{rational-cohomology}).

In case $c(W)=1$, the central extension (\ref{eq:central}) turns out to be
the unique nontrivial central extension of $W$ by $\Z$ (Corollary \ref{cor:unique-cent}).
As is known, the central extension (\ref{eq:central}) corresponds to
the cohomology class $u_\str\in H^2(W,\Z^{\nc})$.
We will construct 2-cocycles of $W$ representing
$u_\str$ (Proposition \ref{thm:cocyle1} and
Theorem \ref{thm:cocycle2}).

Alternatively, Eisermann \cite{MR3205568} 
claimed that
$\adj(\symqdl)$ is isomorphic to the semidirect product $A_n\rtimes\Z$
where $A_n$ is the alternating group on $n$ letters,
but he did not write down the proof. 
We will generalize his result to Coxeter quandles 
(Theorem \ref{thm:comm} and Corollary \ref{gen-eiser}).

In the final section, we deal with root systems.
To each Coxeter system $(W,S)$, one can associate the root system $\rs\subset V$
by using the geometric representation $W\map GL(V)$,  where $V$ is a real vector
space with the basis $\{\alpha_s\mid s\in S\}$.
The root system $\rs$ turns out to
be a rack with the rack operation $\alpha*\beta:=t_\beta(\alpha)$
where $t_\beta\in GL(V)$ is the reflection along $\beta$.
We close this paper by proving $\adj(\rs)\cong \adj(\cq)$
(Theorem \ref{thm:root}).

\begin{notation}
Let $G$ be a group, $g,h\in G$ elements of $G$ and $m\geq 2$ a natural number.
Define
\[
(gh)_m:=\underbrace{ghghg\cdots}_m\in G.
\]
For example, $(gh)_2=gh,(gh)_3=ghg,(gh)_4=ghgh$.
Let $G_\Ab:=G/[G,G]$ be the abelianization of $G$,
$\Ab_G:G\map G_\Ab$ the natural projection,
and write $[g]:=g[G,G]\in G_\Ab$ for $g\in G$.
For a quandle $Q$ and elements $x_k\in Q$ $(k=1,2,3,\dots)$, we denote
\[
x_1\qdl x_2\qdl x_3\qdl\cdots\qdl x_\ell:=
(\cdots((x_1\qdl x_2)\qdl x_3)\qdl\cdots)\qdl x_\ell,
\]
for simplicity.
\end{notation}

\section{Coxeter groups and Coxeter quandles}\label{sec:prel}
\subsection{Coxeter groups}
Let $S$ be a finite set and $m:S\times S\rightarrow \mathbb{N}\cup\{\infty\}$
a map satisfying the following conditions:
\begin{enumerate}
\item $m(s,s)=1$ for all $s\in S$;
\item $2\leq m(s,t)=m(t,s)\leq\infty$ for all distinct $s,t\in S$.
\end{enumerate}
The map $m$ is represented by the \emph{Coxeter graph} $\Gamma$ whose
vertex set is $S$ and whose edges are the unordered pairs $\{s,t\}\subset S$
such that $m(s,t)\geq 3$. The edges with $m(s,t)\geq 4$ are labeled
by the number $m(s,t)$.
The \emph{Coxeter system} associated with $\Gamma$
is the pair $(W,S)$, where $W$ is the group generated by $s\in S$
and the fundamental relations $(st)^{m(s,t)}=1$ $(s,t\in S, m(s,t)<\infty)$:
\begin{equation}\label{Coxeter-pres}
W:=\langle s\in S\ |\ (st)^{m(s,t)}=1 (s,t\in S, m(s,t)<\infty)\rangle.
\end{equation}
The group $W$ is called the \emph{Coxeter group} (of type $\Gamma$),
and elements of $S$ are called \emph{Coxeter generators} of $W$
(also called simple reflections in the literature).
Note that the order of the product $st$ is precisely $m(s,t)$.
In particular, every Coxeter generator $s\in S$ has order $2$.
It is easy to check that the defining relations in (\ref{Coxeter-pres}) 
are equivalent to the following relations
\begin{equation}\label{pres-Coxeter-alt}
s^{2}=1\ (s\in S),\
(st)_{m(s,t)}=(ts)_{m(s,t)}\ (s,t\in S,s\not=t, m(s,t)<\infty).
\end{equation}
Finally,
the \emph{odd subgraph} $\Gamma_{odd}$ is a subgraph of $\Gamma$
whose vertex set is $S$ and whose edges are the unordered pairs $\{s,t\}\subset S$
such that $m(s,t)$ is an odd integer.
We refer
Bj\"orner-Brenti \cite{MR2133266},
Bourbaki \cite{bourbaki}, 
Davis \cite{MR2360474}, and
Humphreys \cite{humphreys}
for further details of Coxeter groups.

\subsection{Conjugacy classes of reflections}
Let
\[
\cq:=\bigcup_{w\in W}w^{-1}Sw
\]
be the set of reflections in $W$ as in the introduction (the underlying set of the Coxeter quandle).
Let $\orbit$ be the set of conjugacy classes of elements of $\cq$ under $W$,
and $\rep$ a complete set of representatives of conjugacy classes.
We may choose such that $\rep\subseteq S$.
Let $\nc$ be the cardinality of $\orbit$.
The following two propositions are well-known and easy to prove
(see Bj\"orner-Brenti \cite{MR2133266}*{Chapter 1, Exercises 16--17}
for instance).
\begin{pro}\label{pro-orbit-odd}
The elements of $\orbit$ are in one-to-one correspondence
with the connected components of $\Gamma_{odd}$.
Consequently, $\nc$ equals to the number of connected
components of $\Gamma_{odd}$. 
\end{pro}\noindent
To be precise, the conjugacy class of $s\in S$ corresponds to
the connected component of $\Gamma_{odd}$ containing $s\in S$.
\begin{pro}\label{pro-coxeter-abel}
$W_\Ab$ is the elementary abelian $2$-group with a basis
$\{[s]\in W_\Ab\mid s\in\rep\}$. In particular,
$W_\Ab\cong (\Z/2)^{\nc}$.
\end{pro}
\subsection{Coxeter quandles}
Now we turn our attention to Coxeter quandles.
Let $\cq$ be the Coxeter quandle associated with $(W,S)$ and
\[
\adj(\cq):=\langle e_x\ (x\in\cq)\mid e_y^{-1}e_x e_y
=e_{x\qdl y}\ (x,y\in\cq)\rangle
\]
the adjoint group of $\cq$.
Observe that
\begin{equation}\label{adj-rel1}
e_y^{-1}e_x e_y=e_{x\qdl y}=e_{y^{-1}xy}
\end{equation}
and
\begin{equation}\label{adj-rel2}
e_ye_xe_y^{-1}=e_{x\qdl y},
\end{equation}
where (\ref{adj-rel2}) follows from $e_y^{-1}e_{x\qdl y}e_y
=e_{x\qdl y\qdl y}=e_x$.
\begin{pro}\label{finite-generation}
$\adj(\cq)$ is generated by $e_s$ $(s\in S)$.
\end{pro}
\begin{proof}
Given $x\in\cq$, we can express $x$ as
$x=(s_1s_2\cdots s_\ell)^{-1}s_0(s_1s_2\cdots s_\ell)$ for some $s_i\in S$
$(0\leq i\leq\ell)$ by the definition of $\cq$. Applying (\ref{adj-rel1}), we have
\[e_x=e_{(s_1s_2\cdots s_\ell)^{-1}s_0(s_1s_2\cdots s_\ell)}=
e_{s_0\qdl s_1\qdl s_2\qdl\cdots\qdl s_\ell} 
=(e_{s_1}e_{s_2}\cdots e_{s_\ell})^{-1}e_{s_0}
(e_{s_1}e_{s_2}\cdots e_{s_\ell}),\]
proving the proposition.
\end{proof}
\begin{pro}\label{pro-abel-adj}
$\adj(\cq)_\Ab$ is the free abelian group with a basis
$\{[e_s]\mid s\in\rep\}$. In particular,
$\adj(\cq)_\Ab\cong \Z^{\nc}$.
\end{pro}
\begin{proof}
By the definition of $\adj(\cq)$, the abelianization $\adj(\cq)_\Ab$ is
generated by $[e_x]$ $(x\in\cq)$ subject to the relations
$[e_x]=[e_{y^{-1}xy}]$, $[e_x][e_y]=[e_y][e_x]$ $(x,y\in\cq)$.
Consequently, $[e_x]=[e_y]\in\adj(\cq)_\Ab$
if and only if $x,y\in\cq$ are conjugate in $W$.
We conclude that $\adj(\cq)_\Ab$ is the free abelian group with a
basis $\{[e_s]\mid s\in\rep\}$.
\end{proof}
Let $\str:\adj(\cq)\map W$ be a surjective homomorphism
defined by $e_x\mapsto x$, which is well-defined by virtue of (\ref{adj-rel1}),
and let $\Cw:=\ker\str$ be its kernel.
\begin{lem}\label{pro-central-subgp}
$\Cw$ is a central subgroup of $\adj(\cq)$.
\end{lem}
\begin{proof}
Given $g\in \Cw$, it suffices to prove $g^{-1}e_x g=e_x$ for all $x\in\cq$.
To do so, set $g=e_{y_1}^{\epsilon_1}e_{y_2}^{\epsilon_2}\cdots
e_{y_\ell}^{\epsilon_\ell}$ $(\epsilon_i\in\{\pm 1\})$.
Applying (\ref{adj-rel1}) and (\ref{adj-rel2}), we have
\[
g^{-1}e_x g=(e_{y_1}^{\epsilon_1}e_{y_2}^{\epsilon_2}
\cdots e_{y_\ell}^{\epsilon_\ell})^{-1}e_x
(e_{y_1}^{\epsilon_1}e_{y_2}^{\epsilon_2}\cdots e_{y_\ell}^{\epsilon_\ell})
=e_{x\qdl y_1\qdl y_2\qdl\cdots\qdl y_\ell}.
\]
Now $x\qdl y_1\qdl y_2\qdl\cdots\qdl y_\ell=
(y_1y_2\cdots y_\ell)^{-1}x(y_1y_2\cdots y_\ell)$,
and the proposition follows from $\str(g)=y_1^{\epsilon_1}y_2^{\epsilon_2}\cdots y_\ell^{\epsilon_\ell}
=y_1y_2\cdots y_\ell=1$.
\end{proof}
\begin{lem}\label{pro-conj-same}
If $x,y\in\cq$ are conjugate in $W$, then $e_x^2=e_y^2\in \Cw$.
\end{lem}
\begin{proof}
It is obvious that $e_x^2\in \Cw$ for all $x\in\cq$.
Since $\Cw$ is a central subgroup, 
$e_x^2=e_y^{-1}e_x^2e_y=(e_y^{-1}e_xe_y)^2=e^2_{y^{-1}xy}$
holds for all $x,y\in\cq$, which implies the lemma.
\end{proof}

\section{Artin groups and the proof of the main result}\label{sec:main}
Now we state the main result of this paper:
\begin{thm}\label{thm-main}
The central subgroup $\Cw$ is the free abelian group with a basis $\{e_s^2\mid s\in\rep\}$.
In particular, $\Cw\cong\Z^{\nc}$.
\end{thm}
As a consequence of Theorem \ref{thm-main}, 
$\adj(\cq)$ fits into a central extension of the form
$0\map\Z^{\nc}\map\adj(\cq)\map W\map 1$
as stated in the introduction.
We begin with the determination of the rank of $\Cw$.
\begin{pro}\label{pro-rank}
$\rank(\Cw)=\nc$.
\end{pro}
\begin{proof}
The central extension $1\map \Cw\map\adj(\cq)\overset{\str}{\map} W\map 1$
yields the following exact sequence for the rational homology of groups:
\[
H_2(W,\Q)\map H_1(\Cw,\Q)_W\map H_1(\adj(\cq),\Q)\map H_1(W,\Q)
\]
(see Brown \cite{brown}*{Corollary VII.6.4}).
Here the co-invariants $H_1(\Cw,\Q)_W$ coincides with $H_1(\Cw,\Q)$ because
$\Cw$ is a central subgroup of $\adj(\cq)$.
It is known that the rational (co)homology of a Coxeter group is trivial
(see Akita \cite[Proposition 5.2]{a-euler} or Davis \cite{MR2360474}*{Theorem 15.1.1}).
As a result, we have an isomorphism
$H_1(\Cw,\Q)\cong H_1(\adj(\cq),\Q)$ and hence we have
\begin{align*}
\rank(\Cw)&=
\dim_\Q (\Cw\otimes\Q)=\dim_\Q H_1(\Cw,\Q)\\
&=\dim_\Q H_1(\adj(\cq),\Q)=\dim_\Q (\adj(\cq)_{\Ab}\otimes\Q)=\nc
\end{align*}
as desired.
\end{proof}

To proceed further, we need the notions of Artin groups and pure Artin groups.
Given a Coxeter system $(W,S)$, the \emph{Artin group} $A_W$ associated with
$(W,S)$ is the group defined by the presentation
\[
A_W:=\langle a_s\ (s\in S)\mid
(a_sa_t)_{m(s,t)}=(a_ta_s)_{m(s,t)}\ (s,t\in S, s\not=t, m(s,t)<\infty)\rangle.
\]
In view of (\ref{pres-Coxeter-alt}),
there is an obvious surjective homomorphism $\pi:A_W\map W$
defined by $a_s\mapsto s$ $(s\in S)$.
The \emph{pure Artin group} $P_W$ associated with $(W,S)$ is defined to be
the kernel of $\pi$ so that there is an extension
\[
1\map P_W\hookrightarrow A_W\overset{\pi}{\map}W\map 1.
\]
In case $W$ is the symmetric group on $n$ letters,
$A_W$ is the braid group of $n$ strands, and $P_W$ is the
pure braid group of $n$ strands.
Artin groups were introduced by Brieskorn-Saito \cite{MR0323910}.
Little is known about the structure of general Artin groups.
Among others, the following questions are still open.
\begin{enumerate}
\item Are Artin groups torsion free?
\item What is the center of Artin groups?
\item Do Artin groups have solvable word problem?
\item Are there finite $K(\pi,1)$-complexes for Artin groups?
\end{enumerate}
See survey articles by Paris \cite{MR2497781,MR3205598,MR3207280}
for further details of Artin groups.
\begin{pro}\label{pro-artin-to-adj}
The assignment $a_s\mapsto e_s$ $(s\in S)$ yields a well-defined surjective
homomorphism $\straa:A_W\map\adj(\cq)$.
\end{pro}
\begin{proof}
As for the well-definedness, it suffices to show
$(e_s e_t)_{m(s,t)}=(e_t e_s)_{m(s,t)}$ for all distinct $s,t\in S$
with $m(s,t)<\infty$. 
Applying the relation $e_xe_y=e_ye_{x\qdl y}$ $(x,y\in\cq)$
repeatedly as in
\[
(e_t e_s)_{m(s,t)}=e_te_se_te_s\cdots 
=e_se_{t\qdl s}e_te_s\cdots
=e_se_te_{t\qdl s\qdl t}e_s\cdots=\cdots,
\]
we obtain
\[
(e_t e_s)_{m(s,t)}=(e_se_t)_{m(s,t)-1}e_x,
\]
where
\begin{equation*}
\begin{split}
x=\overbrace{t\qdl s\qdl t\qdl s\qdl\cdots}^{m(s,t)}
&=(st)_{m(s,t)-1}^{-1}t(st)_{m(s,t)-1}\\
&=(st)_{m(s,t)-1}^{-1}(ts)_{m(s,t)}\\
&=(st)_{m(s,t)-1}^{-1}(st)_{m(s,t)}.
\end{split}
\end{equation*}
In the last equality we used the relation $(st)_{m(s,t)}=(ts)_{m(s,t)}$.
It follows 
that $x$ is the last letter in $(st)_{m(s,t)}$, i.e.
$x=s$ or $x=t$ according as $m(s,t)$ is odd or $m(s,t)$ is even.
We conclude that $(e_se_t)_{m(s,t)-1}e_x=(e_se_t)_{m(s,t)}$
as desired.
Finally, the surjectivity follows from Proposition \ref{finite-generation}.
\end{proof}
As a result, the adjoint group $\adj(\cq)$ 
is an intermediate group between 
a Coxeter group $W$ and the corresponding Artin group $A_W$,
in the sense that the canonical surjection $\pi:A_W\twoheadrightarrow W$, $a_s\mapsto s$
$(s\in S)$
splits into a sequence of surjections
\[A_W\overset{\straa}{\twoheadrightarrow}\adj(\cq)\overset{\str}{\twoheadrightarrow} W,\
a_s\mapsto e_s\mapsto s\ (s\in S).\]

\begin{pro}\label{pres-pure-Artin}
$P_W$ is the normal closure of $\{a_s^2\mid s\in S\}$ in $A_W$.
In other words, $P_W$ is generated by $g^{-1}a_s^2g$ $(s\in S, g\in A_W)$.
\end{pro}
\begin{proof}
Given a Coxeter system $(W,S)$,
let $F(S)$ be the free group on $S$ and put
\[
R:=\{(st)_{m(s,t)}(ts)_{m(s,t)}^{-1}\mid s,t\in S, s\not= t, m(s,t)<\infty\},\
Q:=\{s^2\mid s\in S\}.
\]
Let $N(R),N(Q),N(R\cup S)$ be the normal closure of $R,Q,R\cup Q$ in $F(S)$, respectively.
The third isomorphism theorem yields a short exact sequence of groups
\[
1\map \frac{N(R)N(Q)}{N(R)}\map \frac{F(S)}{N(R)}\overset{p}{\map}
\frac{F(S)}{N(R)N(Q)}\map 1.
\]
Observe that the left term $N(R)N(Q)/N(R)$ is nothing but the normal closure
of $Q$ in $F(S)/N(R)$.
Now $F(S)/N(R)N(Q)=F(S)/N(R\cup Q)=W$ by the definition of $W$ and
$F(S)/N(R)$ is identified with $A_W$ via $s\mapsto a_s$ $(s\in S)$.
Under this identification, the map $p$ is  the canonical
surjection $\pi:A_W\twoheadrightarrow W$ and hence the left term
$N(R)N(Q)/N(R)$ coincides with $P_W$. 
The proposition follows.
\end{proof}
\begin{rem}
Digne-Gomi \cite{MR1831679}*{Corollary 6} obtained a presentation of $P_W$
by using Reidemeister-Schreier method.
Their presentation is infinite whenever $W$ is infinite.
Although Proposition \ref{pres-pure-Artin} may be read off from their presentation,
we wrote down the proof because our proof is much simpler than the arguments in
\cite{MR1831679}.
\end{rem}

\begin{proof}[Proof of Theorem \ref{thm-main}]
Consider the commutative diagram
\[\begin{tikzcd}
1\arrow{r}  & P_W \arrow[r,hook] \arrow[d,"\straa|_{P_W}"]
               &A_W \arrow[d,"\straa"]\arrow[r,"\pi"]&W\arrow[r]\arrow[d,equal]&1\\
1\arrow[r]& \Cw \arrow[r,hook] &\adj(\cq)\arrow[r,"\str"]&W\arrow[r]&1
\end{tikzcd}\]
whose rows are exact. 
Since $\straa:A_W\map\adj(\cq)$ is surjective,
one can check 
that its restriction $\straa|_{P_W}:P_W\map \Cw$ is also surjective.
As $P_W$ is generated by $g^{-1}a_s^2g$ $(s\in S,g\in A_W)$
by Proposition \ref{pres-pure-Artin}, $\Cw$ is generated by
\[
\straa(g^{-1}a_s^2g)=\straa(g)^{-1}e_s^2\straa(g)=e_s^2\ (s\in S, g\in A_W),
\]
where the last equality follows from the fact that
$\Cw$ is central by Lemma \ref{pro-central-subgp}.
Combining with Lemma \ref{pro-conj-same},
we see that $\Cw$ is generated by $\{e_s^2\mid s\in\rep\}$.
Now $\rank (\Cw)=\nc$ by Proposition \ref{pro-rank} and $|\rep|=\nc$
by the defintion of $\rep$,
$\Cw$ must be a free abelian group of rank $\nc$
and $\{e_s^2\mid s\in\rep\}$
must be a basis of $\Cw$.
\end{proof}
As an immediate cosequence of Theorem \ref{thm-main},
we can determine the rational cohomology ring of $\adj(\cq)$:
\begin{cor}\label{rational-cohomology}
For any Coxeter system $(W,S)$, the inclusion $\Cw\hookrightarrow\adj(\cq)$
induces an isomorphism
\[H^*(\adj(\cq),\Q)\overset{\cong}{\map} H^*(\Cw,\Q){\cong}\wedge_\Q(u_1,u_2,\dots,u_{\nc})\]
with $\deg u_k=1$ $(1\leq k\leq \nc)$.
\end{cor}
\begin{proof}
In the Lyndon-Hochschild-Serre spectral sequence
\[
E_2^{pq}=H^p(W,H^q(\Cw,\Q))\Rightarrow H^{p+q}(\adj(\cq),\Q)
\]
associated with the central extension $1\map\Cw\map\adj(\cq)\map W\map 1$,
one has
\[
E_2^{pq}\cong H^p(W,\Q)\otimes_\Q H^q(\Cw,\Q)\cong
\begin{cases}
H^q(\Cw,\Q) & p=0 \\ 0 & p\not=0
\end{cases}\] 
because $H^p(W,\Q)=0$ for $p\not=0$ (see the proof of Proposition \ref{pro-rank}).
The first isomorphism follows immediately.
The second isomorphism follows from the fact $\Cw\cong\Z^{\nc}$.
\end{proof}
\begin{rem}
For any right-angled Coxeter group $W$,
Kishimoto \cite{arXiv:1706.06209}*{Theorem 5.3} determined the mod 2 cohomology ring
of $\adj(\cq)$.
\end{rem}

\section{Construction of $2$-cocycles}\label{sec:cocycle}
Throughout this section, we assume that the reader is familiar with group cohomology
and Coxeter groups.
The central extension 
\begin{equation}\label{cent_ext}
1\map \Cw\map\adj(\cq)\overset{\str}{\map}W\map 1
\end{equation}
corresponds to the cohomology class $u_\str\in H^2(W,\Cw)$
(see Brown \cite{brown}*{\S IV.3} for precise).
In this section, we will construct 2-cocycles representing $u_\str$.
Before doing so,
we claim $u_\str\not=0$.
The claim is equivalent to the following lemma.
\begin{lem}\label{lem:ext-nontriv}
The central extension $1\map \Cw\map\adj(\cq)\overset{\str}{\map}W\map 1$ is nontrivial.
\end{lem}
\begin{proof}
If the central extension is trivial, then $\adj(\cq)\cong \Cw\times W$.
But this is not the case because $\adj(\cq)_\Ab\cong\Z^{\nc}$ 
by Proposition \ref{pro-abel-adj} while
$(\Cw\times W)_\Ab\cong \Cw\times W_\Ab\cong
\Z^{\nc}\times (\Z/2)^{\nc}$ by Proposition \ref{pro-coxeter-abel}.
\end{proof}
Now we invoke the celebrated Matsumoto's theorem:
\begin{thm}[Matsumoto \cite{MR0183818}]\label{thm-matsu}
Let $(W,S)$ be a Coxeter system, $M$ a monoid and $f:S\map M$
a map such that $(f(s)f(t))_{m(s,t)}=(f(t)f(s))_{m(s,t)}$ for all
$s,t\in S$ with $s\not=t$, $m(s,t)<\infty$.
Then there exists a unique map $F:W\map M$ such that
$F(w)=f(s_1)\cdots f(s_k)$ whenever $w=s_1\cdots s_k$ $(s_i\in S)$
is a reduced expression.
\end{thm}
The proof also can be found in 
\cite{bourbaki}*{Chapitre~IV, \S1, Proposition 5} and 
\cite{MR1778802}*{Theorem 1.2.2}.
Define a map $f:S\map\adj(\cq)$ by $s\mapsto e_s$, then $f$ satisfies
the assumption of Theorem \ref{thm-matsu} as in the proof of
Proposition \ref{pro-artin-to-adj}, 
and hence there exists a
unique map $F:W\map\adj(\cq)$ such that 
$F(w)=f(s_1)\cdots f(s_k)=e_{s_1}\cdots e_{s_k}$ whenever
$w=s_1\cdots s_k$ $(s_i\in S)$ is a reduced expression.
It is clear that $F:W\map\adj(\cq)$ is a set-theoretical section of
$\str:\adj(\cq)\map W$.
Define $c:W\times W\map \Cw$ by 
\[c(w_1,w_2)=F(w_1)F(w_2)F(w_1w_2)^{-1}.\]
The standard argument in group cohomology 
(see Brown \cite{brown}*{\S IV.3})
implies the following result:
\begin{pro}\label{thm:cocyle1}
$c$ is a normalized $2$-cocycle and $[c]=u_\str\in H^2(W,\Cw)$.
\end{pro}
\begin{rem}
In case $W$ is the symmetric group of $n$ letters $\mathfrak{S}_n$,
Proposition \ref{thm:cocyle1} was stated in
\cite{MR2799090}*{Remark 3.3}.
\end{rem}
Now we deal with the case $\nc=1$ more precisely.
If $c(W)=1$ then the odd subgraph $\Gamma_{odd}$ of $W$ is connected and
hence $W$ must be irreducible.
All finite irreducible Coxeter groups of type other than
$B_n$ $(n\geq 2)$, $F_4$, $I_2(m)$ ($m$ even) satisfy $c(W)=1$.
Among affine irreducible Coxeter groups, 
those of type $\widetilde{A}_n$ $(n\geq 2)$, $\widetilde{D}_n$ $(n\geq 4)$, 
$\widetilde{E}_6$, $\widetilde{E}_7$ and $\widetilde{E}_8$ 
fulfill $c(W)=1$.
For simplifying the notation, we will identify $\Cw$ with $\Z$ by $e_s^2\mapsto 1$
and denote our central extension by
\[
0\map\Z\map\adj(\cq)\overset{\str}{\map} W\map 1.
\]
\begin{pro}\label{pro-second-coho-conn}
If $\nc=1$ then $H^2(W,\Z)\cong\Z/2$.
\end{pro}
\begin{proof}
A short exact sequence $0\map\Z\hookrightarrow\C\map\C^\times\map 1$ of
abelian groups, where $\C\map\C^\times$ is defined by $z\mapsto\exp(2\pi\sqrt{-1}z)$,
induces the exact sequence
\[
H^1(W,\C)\map H^1(W,\C^\times)\overset{\delta}{\map} H^2(W,\Z)\map H^2(W,\C)
\]
(see Brown \cite{brown}*{Proposition III.6.1}).
It is known that $H^k(W,\C)=0$ for $k>0$ (see the proof of Proposition \ref{pro-rank}),
which implies that the connecting homomorphism
$\delta:H^1(W,\C^\times)\map H^2(W,\Z)$ is an isomorphism.
We claim that $H^1(W,\C^\times)=\Hom(W,\C^\times)\cong\Z/2$.
Indeed, $W$ is generated by $S$ consisting of elements of order $2$,
and all elements of $S$ are mutually conjugate by the assumption $\nc=1$.
Thus $\Hom(W,\C^\times)$ consists of the trivial homomorphism
and the homomorphism $\rho$ defined by $\rho(s)=-1$ $(s\in S)$.
\end{proof}
Combining Lemma \ref{lem:ext-nontriv} and Proposition \ref{pro-second-coho-conn},
we obtain the following corollary.
\begin{cor}\label{cor:unique-cent}
If $\nc=1$ then $0\map\Z\map\adj(\cq)\overset{\str}{\map} W\map 1$
is the unique nontrivial central extension of $W$ by $\Z$.
\end{cor}
In general, given a homomorphism of groups $f:G\map\C^\times$, 
the cohomology class $\delta\! f\in H^2(G,\Z)$,
where $\delta:\Hom(G,\C^\times)\map H^2(G,\Z)$ is the connecting
homomorphism as above, can be described
as follows. For each $g\in G$, choose a branch of $\log(f(g))$.
We argue such that $\log(f(1))=0$.
Define $\tau_f:G\times G\map\Z$ by
\begin{equation}\label{eq-def-cocycle-formula}
\tau_f(g,h)=\frac{1}{2\pi\sqrt{-1}}
\{\log(f(g))+\log(f(h))-\log(f(gh))\}\in\Z.
\end{equation}
By a diagram chase, one can prove the following:
\begin{pro}\label{pro-cocycle-log}
$\tau_f$ is a normalized $2$-cocycle and $[\tau_f]=\delta\! f
\in H^2(G,\Z)$.
\end{pro}
Assuming $\nc=1$, let $\rho:W\map\C^\times$ be the homomorphism defined by
$\rho(s)=-1$ $(s\in S)$ as in the proof of Proposition \ref{pro-second-coho-conn}.
Note that $\rho(w)=(-1)^{\ell(w)}$ $(w\in W)$ where $\ell(w)$ is the
\emph{length} of $w$ (see Humphreys \cite{humphreys}*{\S5.2}).
For each $w\in W$, choose a branch of $\log(\rho(w))$ as
\[
\log(\rho(w))=\begin{cases}
0 & \text{if $\ell(w)$ is even,} \\
\pi\sqrt{-1} & \text{if $\ell(w)$ is odd.}
\end{cases}
\]
Applying (\ref{eq-def-cocycle-formula}),
the corresponding map
$\tau_\rho:W\times W\map\Z$ is given by
\begin{equation}\label{eq-cocycle-coxeter-conn}
\tau_\rho(w_1,w_2)=\begin{cases}
1 & \text{if $\ell(w_1)$ and $\ell(w_2)$ are odd,} \\
0 & \text{otherwise.}
\end{cases}
\end{equation}

Combining Lemma \ref{lem:ext-nontriv}, Corollary \ref{cor:unique-cent}, and
Proposition \ref{pro-cocycle-log},  we obtain the following theorem:
\begin{thm}\label{thm:cocycle2}
If $\nc=1$ then
$\tau_\rho:W\times W\map\Z$ defined by
$(\ref{eq-cocycle-coxeter-conn})$
is a normalized $2$-cocycle and $[\tau_\rho]=u_\str\in H^2(W,\Z)$.
\end{thm}

\section{Commutator subgroups of adjoint groups}\label{sec:commutator}
As was stated in the introduction, Eisermann \cite{MR3205568}*{Example 1.18} claimed that
$\adj(\symqdl)$ is isomorphic to the semidirect product $A_n\rtimes\Z$
where $A_n$ is the alternating group on $n$ letters.
We will generalize his result to Coxeter quandles by showing
the following theorem:
\begin{thm}\label{thm:comm}
$\str:\adj(\cq)\map W$ induces an isomorphism
\[
\str:[\adj(\cq),\adj(\cq)]\overset{\cong}{\map} [W,W].
\]
\end{thm}
\begin{proof}
Consider the following commutative diagram 
with exact rows and columns.
\[\begin{tikzcd}
& & 1\arrow[d] & 1\arrow[d]& \\
& & \Cw \arrow[r, "{\Ab_{\adj(\cq)}}", "\cong"']\arrow[d,hook]
	& \ker\str_\Ab\arrow[d,hook] & \\
1\arrow[r]  & {[\adj(\cq),\adj(\cq)]} \arrow[r,hook] \arrow[d,"\str"]
	&\adj(\cq) \arrow{d}{\str}\arrow{r}{\Ab_{\adj(\cq)}}
	&\adj(\cq)_\Ab\arrow{r}\arrow{d}{\str_\Ab}&1\\
1\arrow[r]& {[W,W]} \arrow[r,hook]
	&W\arrow{r}{\Ab_W}\arrow{d}
	&W_\Ab\arrow{r}\arrow{d}&1\\
& & 1 & 1&
\end{tikzcd}\]
From Proposition \ref{pro-coxeter-abel} and Proposition \ref{pro-abel-adj}, 
we see that $\ker\str_\Ab$ is the free abelian group with a basis
$\{[e_s^2]\mid s\in\rep\}$, which implies that
$\Ab_{\adj(\cq)}:\Cw\map\ker\str_\Ab$ is an isomorphism
because it assigns $[e_s^2]$ to $e_s^2$ $(s\in\rep)$.
Since $\str$ is surjective, it is obvious that $\str( [\adj(\cq),\adj(\cq)])
=[W,W]$. We will show that $\str:[\adj(\cq),\adj(\cq)]\map [W,W]$ is injective.
To do so, let $g\in  [\adj(\cq),\adj(\cq)]$ be an element with $g\in\ker\str=\Cw$.
Then $[g]:=\Ab_{\adj(\cq)}(g)=1$ by the exactness of the middle row.
But it implies $g=1$ because $\Ab_{\adj(\cq)}:\Cw\map\ker\str_\Ab$
is an isomorphism.
\end{proof}
\begin{cor}\label{gen-eiser}
There is a group extension of the form
\[1\map [W,W]\map\adj(\cq)\overset{\Ab_{\adj(\cq)}}{\longrightarrow}\Z^{\nc}\map 1.\]
If $\nc=1$ then  the extension splits and
$\adj(\cq)\cong [W,W]\rtimes\Z$.
\end{cor}
Since $c(\mathfrak{S}_n)=1$ and $[\mathfrak{S}_n,\mathfrak{S}_n]=A_n$,
we recover the claim by Eisermann mentioned above.
\begin{rem}
By using Theorem \ref{thm:comm}, Kishimoto \cite{arXiv:1706.06209}*{Theorem 2.5}
proved that the following commutative square which appeared in the proof
of Theorem \ref{thm:comm} is actually a pull-back.
\[\begin{tikzcd}
\adj(\cq)\arrow{r}{\Ab_{\adj(\cq)}} \arrow[d,"\str"]& \adj(\cq)_{\text{Ab}}\arrow{d}{\str_\Ab}\\
W\arrow{r}{\Ab_W}&W_{\text{Ab}}
\end{tikzcd}\]
\end{rem}

\section{Root systems}
As was pointed out by Andruskiewitsch-Gra\~na \cite{MR1994219}*{\S1.3.5},
Brieskorn \cite{MR975077}*{p.~58, Example 9},
and Fenn-Rourke \cite{MR1194995}*{Example 10}, 
a root system turns out to be a rack.
In this final section, we turn our attention to root systems.
Let us recall the definition and properties of the root system associated with
a Coxeter system $(W,S)$.
See  Bj\"orner-Brenti \cite{MR2133266}*{Chapter 4} 
or Humphreys \cite{humphreys}*{Part II, Chapter 5} for precise.
Let $V$ be the real vector space with a basis $\{\alpha_s\mid s\in S\}$
in one to one correspondence with $S$.
Define a symmetric bilinear form $B(-,-)$ on $V$ by
\[B(\alpha_s,\alpha_t):=-\cos\frac{\pi}{m(s,t)}.\]
Here we understand $B(\alpha_s,\alpha_t)=-1$ in case $m(s,t)=\infty$.
For each $s\in S$, define a reflection $\sigma_s\in GL(V)$ by
\[
\sigma_s(\lambda)=\lambda-2B(\alpha_s,\lambda)\alpha_s.
\]
The \emph{geometric representation} of $W$ (or the canonical
representation in the literature) is the unique homomorphism
$\sigma:W\map GL(V)$ satisfying $s\mapsto\sigma_s$ $(s\in S)$.
The geometric representaion  is faithful and
preserves the form $B(-,-)$.
For simplicity, we may write $w(\lambda)$
in place of $\sigma(w)(\lambda)$ $(w\in W, \lambda\in V)$.
The \emph{root system} $\rs$ associated with $(W,S)$
is defined by
\[\rs:=\{w(\alpha_s)\mid s\in S,w\in W\}.\]
For each root $\alpha\in\rs$, define a reflection $t_\alpha\in GL(V)$ by
$t_\alpha(\lambda):=\lambda-2B(\alpha,\lambda)\alpha$.
Note that if $\alpha=\alpha_s$ $(s\in S)$ then $t_\alpha=\sigma_s$.
\begin{lem}
The root system $\rs$ turns out to be a rack with the rack operation
$\alpha\ast\beta=t_{\beta}(\alpha)$ $(\alpha,\beta\in\rs)$.
\end{lem}
\begin{proof}
It suffices to prove $(\alpha\ast\beta)\ast\gamma=
(\alpha\ast\gamma)\ast(\beta\ast\gamma)$ for $\alpha,\beta,\gamma\in\rs$.
One has
\[
(\alpha\ast\beta)\ast\gamma=t_\gamma(\alpha-2B(\beta,\alpha)\beta)
=t_\gamma(\alpha)-2B(\beta,\alpha)t_\gamma(\beta)
\]
while
\[
(\alpha\ast\gamma)\ast(\beta\ast\gamma)=
t_\gamma(\alpha)\ast t_\gamma(\beta)
=t_\gamma(\alpha)-2B(t_\gamma(\beta),t_\gamma(\alpha))t_\gamma(\beta),
\]
and the assertion follows from $B(t_\gamma(\beta),t_\gamma(\alpha))=B(\beta,\alpha)$.
Note that $\rs$ is not a quandle because $\alpha\ast\alpha=t_\alpha(\alpha)=-\alpha$.
\end{proof}
Now let $\cq$ be the Coxeter quandle associated with 
a Coxeter system $(W,S)$ as before, and
define a map $p:\rs\to\cq$ by $\alpha\mapsto\ t_\alpha$.
Here we identify elements of $\sigma(W)$ with those of $W$, 
which is possible since $\sigma$ is injective.
Note that $t_\alpha$ belongs to $Q_W$ because
$t_{\alpha}=wsw^{-1}$ holds for $\alpha=w(\alpha_s)$,
and that the map $p$ is surjective and two-to-one
(see Humphreys \cite{humphreys}*{p.~116} or
Bj\"orner-Brenti \cite{MR2133266}*{p.~104}).
Indeed, $p$ maps $\pm\alpha\in\rs$
to the same element $t_\alpha=t_{-\alpha}$.
\begin{lem}\label{lem:morphism}
The map $p\colon\rs\to\cq$ is a morphism of racks.
\end{lem}
\begin{proof}
In general, $wt_\alpha w^{-1}=t_{w(\alpha)}$ holds
for $\alpha\in\rs$ and $w\in W$
(see Bj\"orner-Brenti \cite{MR2133266}*{p.~104}). 
Now for $\alpha,\beta\in\rs$, set $\gamma:=t_\beta(\alpha)$ and we have
\[
p(\alpha\ast\beta)=p(\gamma)=t_\gamma
=t_\beta t_\alpha t_\beta^{-1}=t_\beta^{-1} t_\alpha t_\beta=
t_\alpha\ast t_\beta=p(\alpha)\ast p(\beta)
\]
as desired.
\end{proof}

We close this paper by proving the following result:
\begin{thm}\label{thm:root}
The morphism $p\colon\rs\to\cq$ of racks induces an isomorphism
$\adj(\rs)\cong\adj(\cq)$.
\end{thm}
\begin{proof}
Recall that $\adj(\rs)$ is defined by the presentation
\[
\adj(\rs)=\langle e_{\alpha}\ (\alpha\in\rs)\mid
e_{\alpha\ast\beta}=e_\beta^{-1}e_\alpha e_\beta\rangle.
\]
For each $\alpha\in\rs$, we have $e_{\alpha\ast\alpha}=e_{-\alpha}$
while $e_\alpha^{-1}e_\alpha e_\alpha=e_\alpha$,
which proves that $e_{-\alpha}=e_\alpha$.
On the other hand, as $p:\rs\to\cq$ is surjective,
$\adj(\cq)$ may be defined as
\[
\adj(\cq)=\langle e_{p(\alpha)}\ (\alpha\in\rs)\mid
e_{p(\alpha)\ast p(\beta)}=e_{p(\beta)}^{-1}e_{p(\alpha)} e_{p(\beta)}\rangle.
\]
Since $p^{-1}(\alpha)=\{\pm\alpha\}$
and $p(\alpha\ast\beta)=p(\alpha)\ast p(\beta)$  by Lemma \ref{lem:morphism},
the assignment $e_\alpha\mapsto e_{p(\alpha)}$ $(\alpha\in\rs)$
induces an isomorphism $\adj(\rs)\map\adj(\cq)$.
\end{proof}
\begin{rem}
For a (crystallographic) root system $\Phi$ and its Weyl group $W(\Phi)$,
one can prove that $\adj(\Phi)\cong\adj(Q_{W(\Phi)})$ in a similar way.
\end{rem}

\begin{ack}
The author thanks Daisuke Kishimoto and Takefumi Nosaka 
for valuable comments and discussions with the author. 
He also thanks Ye Liu for informing the
paper Digne-Gomi \cite{MR1831679} and careful reading of
drafts of this paper.
This work was partially supported by JSPS KAKENHI Grant Number
26400077, and by the Research Institute for Mathematical Sciences, a Joint
Usage/Research Center located in Kyoto University.
\end{ack}


\begin{bibdiv}
\begin{biblist}

\bib{a-euler}{article}{
   author={Akita, Toshiyuki},
   title={Euler characteristics of Coxeter groups, PL-triangulations of
   closed manifolds, and cohomology of subgroups of Artin groups},
   journal={J. London Math. Soc. (2)},
   volume={61},
   date={2000},
   number={3},
   pages={721--736},
   issn={0024-6107},
   review={\MR{1766100 (2001f:20080)}},
   doi={10.1112/S0024610700008693},
}

\bib{MR2799090}{article}{
   author={Andruskiewitsch, N.},
   author={Fantino, F.},
   author={Garc{\'{\i}}a, G. A.},
   author={Vendramin, L.},
   title={On Nichols algebras associated to simple racks},
   conference={
      title={Groups, algebras and applications},
   },
   book={
      series={Contemp. Math.},
      volume={537},
      publisher={Amer. Math. Soc., Providence, RI},
   },
   date={2011},
   pages={31--56},
   review={\MR{2799090 (2012g:16065)}},
   doi={10.1090/conm/537/10565},
}

\bib{MR1994219}{article}{
   author={Andruskiewitsch, Nicol\'as},
   author={Gra\~na, Mat\'\i as},
   title={From racks to pointed Hopf algebras},
   journal={Adv. Math.},
   volume={178},
   date={2003},
   number={2},
   pages={177--243},
   issn={0001-8708},
   review={\MR{1994219}},
   doi={10.1016/S0001-8708(02)00071-3},
}

\bib{MR2133266}{book}{
   author={Bj\"orner, Anders},
   author={Brenti, Francesco},
   title={Combinatorics of Coxeter groups},
   series={Graduate Texts in Mathematics},
   volume={231},
   publisher={Springer, New York},
   date={2005},
   pages={xiv+363},
   isbn={978-3540-442387},
   isbn={3-540-44238-3},
   review={\MR{2133266}},
}

\bib{bourbaki}{book}{
   author={Bourbaki, Nicolas},
   title={\'El\'ements de math\'ematique},
   language={French},
   note={Groupes et alg\`ebres de Lie. Chapitres 4, 5 et 6. [Lie groups and
   Lie algebras. Chapters 4, 5 and 6]},
   publisher={Masson},
   place={Paris},
   date={1981},
   pages={290},
   isbn={2-225-76076-4},
   review={\MR{647314 (83g:17001)}},
}



\bib{MR975077}{article}{
   author={Brieskorn, E.},
   title={Automorphic sets and braids and singularities},
   conference={
      title={Braids},
      address={Santa Cruz, CA},
      date={1986},
   },
   book={
      series={Contemp. Math.},
      volume={78},
      publisher={Amer. Math. Soc., Providence, RI},
   },
   date={1988},
   pages={45--115},
   review={\MR{975077}},
}

\bib{MR0323910}{article}{
   author={Brieskorn, Egbert},
   author={Saito, Kyoji},
   title={Artin-Gruppen und Coxeter-Gruppen},
   language={German},
   journal={Invent. Math.},
   volume={17},
   date={1972},
   pages={245--271},
   issn={0020-9910},
   review={\MR{0323910}},
   doi={10.1007/BF01406235},
}

\bib{brown}{book}{
   author={Brown, Kenneth S.},
   title={Cohomology of groups},
   series={Graduate Texts in Mathematics},
   volume={87},
   publisher={Springer-Verlag},
   place={New York},
   date={1982},
   pages={x+306},
   isbn={0-387-90688-6},
   review={\MR{672956 (83k:20002)}},
}


\bib{arXiv:1011.1587}{article}{
title={The adjoint group of an Alexander quandle},
author={Clauwens, F.J.-B.J.},
date={2011},
eprint={http://arxiv.org/abs/1011.1587}
}

\bib{MR2360474}{book}{
   author={Davis, Michael W.},
   title={The geometry and topology of Coxeter groups},
   series={London Mathematical Society Monographs Series},
   volume={32},
   publisher={Princeton University Press, Princeton, NJ},
   date={2008},
   pages={xvi+584},
   isbn={978-0-691-13138-2},
   isbn={0-691-13138-4},
   review={\MR{2360474 (2008k:20091)}},
}

\bib{MR1831679}{article}{
   author={Digne, F.},
   author={Gomi, Y.},
   title={Presentation of pure braid groups},
   journal={J. Knot Theory Ramifications},
   volume={10},
   date={2001},
   number={4},
   pages={609--623},
   issn={0218-2165},
   review={\MR{1831679}},
   doi={10.1142/S0218216501001037},
}


\bib{MR3205568}{article}{
   author={Eisermann, Michael},
   title={Quandle coverings and their Galois correspondence},
   journal={Fund. Math.},
   volume={225},
   date={2014},
   number={1},
   pages={103--168},
   issn={0016-2736},
   review={\MR{3205568}},
   doi={10.4064/fm225-1-7},
}

\bib{MR1194995}{article}{
   author={Fenn, Roger},
   author={Rourke, Colin},
   title={Racks and links in codimension two},
   journal={J. Knot Theory Ramifications},
   volume={1},
   date={1992},
   number={4},
   pages={343--406},
   issn={0218-2165},
   review={\MR{1194995}},
}



\bib{MR1778802}{book}{
   author={Geck, Meinolf},
   author={Pfeiffer, G{\"o}tz},
   title={Characters of finite Coxeter groups and Iwahori-Hecke algebras},
   series={London Mathematical Society Monographs. New Series},
   volume={21},
   publisher={The Clarendon Press, Oxford University Press, New York},
   date={2000},
   pages={xvi+446},
   isbn={0-19-850250-8},
   review={\MR{1778802}},
}


\bib{humphreys}{book}{
   author={Humphreys, James E.},
   title={Reflection groups and Coxeter groups},
   series={Cambridge Studies in Advanced Mathematics},
   volume={29},
   publisher={Cambridge University Press},
   place={Cambridge},
   date={1990},
   pages={xii+204},
   isbn={0-521-37510-X},
   review={\MR{1066460 (92h:20002)}},
}

\bib{arXiv:1706.06209}{article}{
title={Right-angled Coxeter quandles and polyhedral products},
author={Kishimoto, Daisuke},
date={2017},
eprint={http://arxiv.org/abs/1706.06209}
}

\bib{MR0183818}{article}{
   author={Matsumoto, Hideya},
   title={G\'en\'erateurs et relations des groupes de Weyl
   g\'en\'eralis\'es},
   language={French},
   journal={C. R. Acad. Sci. Paris},
   volume={258},
   date={1964},
   pages={3419--3422},
   review={\MR{0183818}},
}




\bib{arXiv:1505.03077}{article}{
   author={Nosaka, Takefumi},
   title={Central extensions of groups and adjoint groups of quandles},
   conference={
      title={Geometry and analysis of discrete groups and hyperbolic spaces},
   },
   book={
      series={RIMS K\^{o}ky\^{u}roku Bessatsu, B66},
      publisher={Res. Inst. Math. Sci. (RIMS), Kyoto},
   },
   date={2017},
   pages={167--184},
   review={\MR{3821082}},
}

\bib{nosaka-book}{book}{
   author={Nosaka, Takefumi},
   title={Quandles and topological pairs},
   series={SpringerBriefs in Mathematics},
   note={Symmetry, knots, and cohomology},
   publisher={Springer, Singapore},
   date={2017},
   pages={ix+136},
   isbn={978-981-10-6792-1},
   isbn={978-981-10-6793-8},
   review={\MR{3729413}},
   doi={10.1007/978-981-10-6793-8},
}


\bib{MR2497781}{article}{
   author={Paris, Luis},
   title={Braid groups and Artin groups},
   conference={
      title={Handbook of Teichm\"uller theory. Vol. II},
   },
   book={
      series={IRMA Lect. Math. Theor. Phys.},
      volume={13},
      publisher={Eur. Math. Soc., Z\"urich},
   },
   date={2009},
   pages={389--451},
   review={\MR{2497781}},
   doi={10.4171/055-1/12},
}

\bib{MR3205598}{article}{
   author={Paris, Luis},
   title={$K(\pi,1)$ conjecture for Artin groups},
   language={English, with English and French summaries},
   journal={Ann. Fac. Sci. Toulouse Math. (6)},
   volume={23},
   date={2014},
   number={2},
   pages={361--415},
   issn={0240-2963},
   review={\MR{3205598}},
   doi={10.5802/afst.1411},
}

\bib{MR3207280}{article}{
   author={Paris, Luis},
   title={Lectures on Artin groups and the $K(\pi,1)$ conjecture},
   conference={
      title={Groups of exceptional type, Coxeter groups and related
      geometries},
   },
   book={
      series={Springer Proc. Math. Stat.},
      volume={82},
      publisher={Springer, New Delhi},
   },
   date={2014},
   pages={239--257},
   review={\MR{3207280}},
   doi={10.1007/978-81-322-1814-2\_13},
}

\end{biblist}
\end{bibdiv}
\end{document}